\documentclass[12pt]{amsart}
\usepackage[utf8]{inputenc}
\usepackage{amsmath}
\usepackage{amsfonts}
\usepackage{amssymb}
\usepackage{amsthm}
\usepackage[all]{xy}
\usepackage{color}
\usepackage{amscd}
\usepackage[mathscr]{eucal}
\usepackage{verbatim}
\usepackage{tikz-cd}
\usepackage{hyperref}
\def\C{{\mathbb{C}}}

\newtheorem{theorem}{Theorem}[section]
\newtheorem{lemma}[theorem]{Lemma}
\newtheorem{proposition}[theorem]{Proposition}
\newtheorem{corollary}[theorem]{Corollary}
\newtheorem{definition}[theorem]{Definition}

\theoremstyle{definition} 
\newtheorem{remark}[theorem]{Remark}
\newtheorem{notation}[theorem]{Notation}

\numberwithin{equation}{section}

\title[Parabolic Bundles on Nodal Curves] {A Note on Parabolic Bundles on Nodal Curves}

\author[ C. Arusha]{ C. Arusha}
\address{ Indian Institute of Science Education and Research, Bhopal}
\email{arushnmaths@hotmail.com}

\author[Sanjay Kumar Singh]{Sanjay Kumar Singh}
\address{Indian Institute of Science Education and Research, Bhopal}
\email{sanjayks@iiserb.ac.in}

\begin{document}



\date{}

\begin{abstract}
Mehta and Seshadri have proved that the set of equivalence
classes of irreducible unitary representations of the fundamental group of a punctured compact Riemann surface,
can be identified with equivalence classes of stable parabolic bundles of parabolic degree zero on
the compact Riemann surface. In this note, we discuss the Mehta-Seshadri correspondence over an irreducible projective curve with at most nodes as singularities. 
\end{abstract}


\maketitle

\section{Introduction}

Let $X$ be a nonsingular irreducible projective algebraic curve defined over $\mathbb{C}$. 
A representation $\rho$ of $\pi_1(X)$, the fundamental group of $X$, into ${\rm GL}(n,\C)$ induces a 
vector bundle of rank $n$ and degree zero on $X$. Weil \cite{A} proved that a vector bundle on $X$ 
arises from a representation if and only if it is a direct sum of indecomposable vector bundles of 
degree zero. Narasimhan and Seshadri \cite{NS} characterised the vector bundles that arises from 
unitary vector bundles by showing that a vector bundle is associated to a unitary representation 
if and only if it is a direct sum of stable bundles of degree zero.

In \cite{VMCS}, Mehta and Seshadri constructed the moduli space of parabolic bundles and proved that 
the equivalence classes of stable parabolic bundles of degree zero on a compact Riemann surface can be 
identified with the set of equivalence classes of irreducible unitary representations of the fundamental 
group of a punctured Riemann surface. In \cite{IB} and \cite{IBM}, the authors prove the parabolic analogue 
of Weil's criterion. Generalising the notion of parabolic bundles, Bhosle introduced generalised parabolic 
bundles in \cite{UB4}. In \cite{UB3}, Bhosle addressed a question of Hitchin and Horrocks on relating 
generalised parabolic bundles to representations of some group. She provided a partial answer to this 
question and also proved that the Narasimhan-Seshadri correspondence is not true for an irreducible nodal curve.

Later, Narasimhan and Ramadas \cite{MSTR}  studied  generalised parabolic sheaves with parabolic structures
at finitely many points and obtained a moduli space for rank two torsion free sheaves endowed with these 
structures. Sun \cite{XS} obtained similar results for torsion free sheaves of arbitrary ranks. This 
motivates one to study the analogue of Mehta-Seshadri correspondence for nodal curves.



This note primarily aims to generalise the results of \cite{UB3} to 
the context of generalised parabolic bundles with parabolic structures. More precisely, 
let $Y$ denote an irreducible nodal curve with arithmetic genus $g$ and let $p:X\to Y$ be its 
normalisation with genus $g(X)$, we prove

\begin{theorem}\label{Result1}
	A parabolic bundle $F_*$, with a parabolic structure at $P$, of parabolic degree zero on an 
	irreducible nodal curve $Y$, is associated to a representation $\zeta$ of $\pi_1(Y-P)$ if and only if every direct summand of its pullback to the normalisation $X$ of $Y$ has parabolic degree zero.
	
	If $\zeta_{X}$, the restriction of $\zeta$ to $\pi_1(X-P)$, is unitary (resp. irreducible unitary) then $F_*$ is parabolic semistable (resp. stable).
\end{theorem}

Our secondary aim is to show that the stable parabolic bundles that arise from representations form a big open set. 
We use the codimension computations carried out for in \cite[Section 7.1]{AUS} and \cite[Proposition 5.1]{XS} to prove the
following theorem.

Let $U_Y$ denote the moduli space of $S$-equivalence classes of semistable parabolic sheaves of parabolic degree zero and  let  $U'_Y$  be the open dense subvariety of $U_Y$ corresponding to parabolic vector bundles.
We denote by $\overline{U}_Y'^s$ the  subset of $U'_Y$ 
corresponding to parabolic bundles $F_*$ such that $p^*(F_*)$ is parabolic stable.

\begin{theorem}\label{Result3}
	Let $Y$ be a complex nodal curve with $g(X) \ge 2$ and let $P\in Y$ be a smooth point. 
	The subset of $U'_Y$ consisting of vector bundles which arise from irreducible unitary 
	representations of the  fundamental group of $\pi_1(Y-P)$ has complement of codimension at least $2$.
\end{theorem}


In Section 2, we discuss all the required preliminaries and known results. In Section 3, we prove Theorem 
\ref{Result1}. We discuss the proof of Theorem \ref{Result3} in Section 4.

\section{Preliminaries}

\subsection{Parabolic Sheaves}

Let $Y$ be an irreducible projective curve over $\mathbb{C}$. Let $I$ be a fixed finite set of smooth  points of $Y$.


\begin{definition}\label{parabolic}
A parabolic sheaf on $Y$ with parabolic structure over $I$ is a torsion free sheaf
$E$ on $Y$ together with the following data at each $p \in I$,

\begin{enumerate}
		\item  a flag on the fibre $E_p$,
		\begin{equation*}
			E_p=F_0(E_p)\supset F_1(E_p)\supset\dots \supset F_{r}(E_p)=0.
		\end{equation*}
		\item and real weights $\alpha=(\alpha_1,\dots, \alpha_{r})$ attached to  
		$F_0(E_p),F_1(E_p), \dots, F_{r}(E_p)$ such that $0\leq\alpha_1<\alpha_2<\dots<\alpha_{r}<1$.		
	\end{enumerate}
\end{definition}
\begin{remark}
	\hfill
 \begin{enumerate}
   \item A quasi-parabolic structure on $E$ over $I$ is simply the condition $(1)$ of Definition \ref{parabolic} for each $p\in I$.
  \item Let $l_i=\dim F_{i}(E_p)$ and $k_i=l_i-l_{i+1}$. Then $l=(l_0,l_1, \dots, l_r)$ is called the flag type of $E$ at $p$.
  The numbers $k_i$ are called the multiplicities of $\alpha_i$ at $p$.
  \item  We say that $E$ is a torsion free sheaf  with a parabolic structure (or a parabolic sheaf) over $I$ or at $p$ if $I=\{p\}$. Also, to distinguish a parabolic  sheaf from its underlying torsion free sheaf, we denote it  by $E_*$ when $E$ is the underlying torsion free sheaf. And, if the underlying sheaf $E$ is locally free, we call it a parabolic vector bundle or simply parabolic bundle.
 \end{enumerate}

\end{remark}


%
%
%
%

\begin{definition}
	Let $E, F$ be  torsion free sheaves on $Y$ with parabolic structures 
over $I$.
	We say $F_*$ is  a parabolic subsheaf of $E_*$ if 
	\begin{enumerate}
		\item $F$ is a subsheaf of $E$ such that $E/F$ is torsion free,
		\item At all parabolic points $p\in I$, given $i_0$, $F_{i_0}(F_p)\subset F_j(E_p)$ for some $j$ and
		\item Let $j_0$ be such that $F_{i_0}(F_p)\subset F_{j_0}(E_p)$ and $F_{i_0}(F_p)\not\subset F_{j_0+1}(E_p)$, 
		then  $\beta_{i_0}=\alpha_{j_0}$ where  $\beta=(\beta_1,\dots,\beta_s)$ and $\alpha=(\alpha_1,\dots,\alpha_r)$ 
		denote the weights of $F_*$ and $E_*$ at $p$, respectively. 
	\end{enumerate}
\end{definition}

Let $E$ be a torsion free sheaf on $Y$ with parabolic structures over $I$ and $F\subset E$ be subsheaf of $E$ such that $E/F$ is torsion free.
Then the parabolic structure  on  $E$ induces a parabolic structure on $F$ and for each $p\in I$, we obtain a flag
$$F_p=F_0(F_p)\supset F_1(F_p):=F_1(E_p)\cap F_p\supset\dots\supset F_r(F_p):=F_r(E_p)\cap F_p\,.$$
The weight associated to $F_k(F_p)$, say $\beta_k$, is defined by $\beta_k:=\alpha_i$ where 
$F_i(E_p)$ is the smallest subspace such that $F_k(F_p)\subset F_i(E_p)$ and $\alpha=(\alpha_1,\dots,\alpha_r)$ is the weight of $E$ at $p$.

\begin{definition}
Let $E, G$ be  torsion free sheaves on $Y$ with parabolic structures over $I$. We say  $G_*$ is a parabolic quotient sheaf of $E_*$ if
	\begin{enumerate}
		\item there is a homomorphism $\phi:E\rightarrow G$ such that $G$ is a quotient sheaf of $E$ 
		under the homomorphism $\phi$, and
		\item at every parabolic vertex $p\in I $, for $1\leq k_0\leq n$, if $j_0$ 
		is such that $\phi(F_{j_0}(E_p))=F_{k_0}(G_p)$ and $\phi(F_{j_0+1}(E_p))\neq F_{k_0}(G_p)$,
		then $\alpha_{j_0}=\gamma_{k_0}$ where $\alpha=(\alpha_1,\dots,\alpha_r)$ and 
		$\gamma=(\gamma_1,\dots,\gamma_s)$ 
		denote the weights of $E_*$ and $G_*$ at $p$, respectively..
	\end{enumerate}
\end{definition}

 Let $E$ be a torsion free sheaf on $X$ with parabolic structures over $I$ and $\phi:E\to G$ be a quotient sheaf 
of $E$. Then $\phi_p:E_p\to G_p$ induces a flag on $G_p$ for each $p\in I$. The weight on $F_j(G_p)$ is $\alpha_i$ 
where $i$ is the largest integer such that $\phi(F_i(E_p))=F_j(F_P)$ and $\alpha=(\alpha_1,\dots,\alpha_r)$ is the weight of $E$ at $p$.
This induces a parabolic structure on $G$ which makes $G_*$ a parabolic quotient sheaf of $E_*$.

\begin{remark}\label{inpardirect}

Let $E_*$ and $F_*$ be parabolic vector bundles on $Y$ with parabolic structures
 over $I$ such that we have an exact sequence
$$0\to E\to G\to F\to 0$$ of vector bundles on $Y$. We can give a canonical parabolic structure on $G$, which can be  
constructed as follows: For any point $p\in I$, we choose an isomorphism 
$$G_p\cong E_{p}\oplus F_{p},$$
compatible with the previous exact sequence. Let
$$A=\{\alpha_1,\alpha_2,\dots, \alpha_r \}\cup \{\beta_1,\beta_2, \dots, \beta_s\}$$
where $\alpha_1,\alpha_2, \dots, \alpha_r $ and $\beta_1,\beta_2, \dots, \beta_s$ are weights of $E$ and $F$ at $p$,
respectively.
 Write $A= \{\gamma_1, \gamma_2, \dots, \gamma_t\}$ where  $\gamma_i$s  are arranged in the ascending order.

\noindent For $1\leq i\leq t$, we have
$$F_i(G_p)=F_j(E_{p})\oplus F_k(F_{p}),$$
where $j$ (resp. $k$) is the smallest integer such that $\gamma_i\leq \alpha_j$ (resp. $\gamma_i\leq \beta_k)$). Thus, we get a flag of $G$ of length $t$ at $p$ and  weights $\gamma_1, \dots, \gamma_t$.
One can easily verify that $E_*$ (resp. $F_*$) is a parabolic subbundle (resp. parabolic quotient bundle) of 
$G_*$ with its canonical structure. For more details, see \cite{CSS}.

\end{remark}

\begin{definition}\label{pardegdefn}
	Let $E_*$ be a parabolic  sheaf on $Y$ with parabolic structures  over $I$. Then the 
	parabolic degree is defined as
	\begin{equation*}
		{\rm par}\deg E_*=\deg E+\sum_{p\in I}(\alpha_1(p)k_1(p)+\dots + \alpha_{r_p}(p)k_{r_p}(p))
	\end{equation*}
	where $\alpha(p):=(\alpha_1(x),\alpha_2(p),\dots, \alpha_{r_p}(p)) $ are the 
	weights of $E$ at $p\in I$ with multiplicities $k(p):=(k_1(p),k_2(p),\dots, k_{r_p}(p))$ and $\deg E$ is 
	the degree of underlying torsion free sheaf $E$. 
\end{definition}

\begin{definition}\label{parstable}
	A parabolic  sheaf $E_*$ on $Y$ is said to be parabolic stable (resp. parabolic semistable) if, for every parabolic subsheaf $E'_*$ of $E_*$, we have
	\begin{equation*}
		\frac{{\rm par}\deg E'_*}{{\rm rk }\ E'}<({\rm resp.}\ \leq)\frac{{\rm par }\deg E_*}{\text{rk }E}.
	\end{equation*}
\end{definition}

\begin{remark}
It is easy to verify that a parabolic  sheaf $E_*$ on $Y$ is parabolic stable (resp. parabolic semistable)
if, for every torsion free quotient sheaf $G_*$ of $E_*$, we have
	\begin{equation*}
		\frac{{\rm par}\deg E_*}{{\rm rk }\ E}<({\rm resp.}\ \leq)\frac{{\rm par }\deg G_*}{\text{rk }G}.
	\end{equation*}
 
\end{remark}

\begin{definition}
 A parabolic vector bundle $E_*$  is $(l,m)$-stable (resp. $(l,m)$-semistable) if for every proper
subbundle $F_*$ of $E_*$, the inequality

\begin{equation*}
		\frac{{\rm par}\deg F_*+l}{{\rm rk }\ F}<({\rm resp.}\ \leq)\frac{{\rm par }\deg E_*+l-\text{rk }E}{\text{rk }E}.
	\end{equation*}
	holds.
	
\end{definition}

 \begin{remark}\label{PB}
 Given real numbers $0\leq\alpha_1<\dots< \alpha_r< 1$, integers $(l_1,\dots, l_r)$ and vector bundles $E$ and $F$ over $Y$, we can always choose a suitable  parabolic structure on $E$ and $F$ at $p\in Y$, using Remark \ref{inpardirect}, such that $G=E\oplus F$ 
 is a parabolic bundle with weights $\alpha_1<\dots, \alpha_r$ and flag type $(l_0, l_1,\dots, l_r)$ at $p$. 
 Hence, we can always find  parabolic bundles $E_*$ and $F_*$ such that 
 $0< {\rm par}\deg E_*\leq 1, -1\leq {\rm par}\deg F_*<0$, and  ${\rm par}\deg (G_*)=0$.
 \end{remark}

\subsection{Topological Results}\hfill

We also recall some topological facts related to nodal curve $Y$ over $\mathbb{C}$. 
We use $\pi_1()$ to denote the topological fundamental group.

\begin{theorem}\label{fungrp}
	Let $Z$ be a topological space and  $Z=Z_1\cup Z_2$ is the union of two arcwise connected open subsets $Z_1$ and $Z_2$. If 
	$Z_1\cap Z_2$ $=A\cup B$ is the union of two arcwise connected non-empty disjoint open sets 
	$A$ and $B$ and if $Z_2$, $A$, and $B$ are simply connected, then
	\begin{equation*}
		\pi_1(Z)\cong \pi_1(Z_1)* \mathbb{Z}.
	\end{equation*}
\end{theorem}

\begin{proof}
	See \cite[Theorem 3.1]{DR}.
\end{proof}

\begin{corollary}\label{fungrpcor}
	Let $Y$ be a nodal curve with one node and let $X$ be its normalisation. Let $P\in Y$ be a smooth point. Let $P$ also denote its unique preimage under the normalisation map. Then $\pi_1(Y-P)\cong\pi_1(X-P)*\mathbb{Z}$.
\end{corollary}
\begin{proof}
 The proof is similar to that of \cite[ Result 1.7]{UB3}; we outline it here for the sake of
completeness. Let $y_0$ denote the nodal point of $Y$ and let $\{x_1, x_2\}$ denote its preimages
under the normalisation map $p: X \to Y\,.$ Let $Y'$  be the curve obtained from $X-P$ by
attaching a solid handle $h$ from $x_1$ to $x_2$. Then one obtains $Y-P$ from $Y'$ by contracting
the handle $h$ to a point $y_0$, and hence $\pi_1(Y-P) \cong \pi_1(Y')$. Observe that $Y'$
can be expressed as $Y'=X-P\cup h$. Applying Theorem \ref{fungrpcor} with a suitable choice of open subsets
$Z_1$ and $Z_2$,  we get the result.
\end{proof}

\begin{remark}
 If $Y$ has $n$ nodes, then $\pi_1(Y-P)\cong\pi_1(X-P)*\mathbb{Z}*\overset{n}{\cdots}*\mathbb{Z}$.
\end{remark}

\subsection{Representations and Parabolic bundles}\hfill\label{rep}

There is an equivalence of categories between the category of smooth irreducible projective algebraic curves over $\C$ (with non-constant regular maps as morphisms) and the category of compact Riemann surfaces (with non-constant holomorphic maps as morphisms). 
Hence, one can use complex analytic methods in algebraic geometry and algebro-geometric methods in complex analysis when studying these objects.  In this note, we use this fact and consider $X$ as a compact Riemann surface and smooth projective curve, interchangeably.

In the section, we recall the construction of parabolic bundles from representations of the fundamental group of the punctured Riemann surface given by Mehta and Seshadri \cite{VMCS}. We follow their notations for convenience.

Let $X$ be a compact Riemann surface of genus $g(X)\geq2$. 
Then the upper half-plane $H=\{z\in \mathbb{C}\vert \ {\rm Im}(z)>0\}$ is the 
universal covering of $X$. Let $P\in X$ be a point of $X$. 
Consider the punctured Riemann surface $X-P$, its fundamental 
group $\Gamma=\pi_1(X-P)$ and its universal covering $H$. Adding the set $S_P$ of 
parabolic points of $\pi_1(X-P)$ to $H$, we get $H^+=H\cup S_P$, and 
we obtain $X$ as the quotient of $H^+$ by $\pi_1(X-P)$ (i.e., $X=H^+/\Gamma$). For $Q\in S_P$, 
the isotropy subgroup $\Gamma_Q$ of $Q$ for the action of $\pi_1(X-P)$ on $H^+$ is 
isomorphic to $\mathbb{Z}$; let $\tau_Q$ denote its generator. 
Let $\rho:\pi_1(X-P)\to GL(n)$ be a representation of $\pi_1(X-P)$ such that $\rho(\tau_P)=\exp(A_P)$ 
where $A_P$ is a diagonalizable matrix. 
Then the vector bundle $E_{\rho}$ on $X-P$ associated to $\rho$ can be extended to 
a vector bundle $E$ on $X$ using bounded functions on a neighborhood $U$ around $P$. 
Also, $E$ is endowed with a natural parabolic structure given by $A_P$ as follows. 
Since $A_P$ can be diagonalized, we get a diagonal matrix
$$
D=
\begin{pmatrix}
	(\alpha_1)I_{k_1} & & \\
	& \ddots & \\
	& & (\alpha_r)I_{k_r}
\end{pmatrix}
$$ 
where $I_{k_i}$ denote the identity matrix of rank $k_i$ and $0\leq\alpha_1<\dots<\alpha_r<1$. The matrix $D$ 
determines a flag of $E_P$ stabilised by it, giving $E$ the quasiparabolic structure at $P$. The weights of $E$ 
at $P$ are $(\alpha_1,\dots, \alpha_r)$ with multiplicities $(k_1,\dots,k_r)$, respectively, so that $E$ is 
naturally $E_*$.

 \begin{theorem}(Mehta-Seshadri Theorem)\label{MS}
 Let $X$ be a smooth projective curve, and $S$ be a finite set of smooth points of $X$.  
 There exist a bijective correspondence between the set of the isomorphism classes of stable parabolic 
 vector bundles of rank $n$ over $X$ with parabolic structures over $S$  and the set of equivalence classes 
 of irreducible unitary representation of $\pi_1(X-S)$.
 \end{theorem}

\subsection{Holomorphic Connections on Parabolic Bundles}\hfill

Let X be an irreducible smooth projective curve defined over $\mathbb{C}$. 
Let $$S = \{P_1, \dots, P_n\} \subset X$$ 
be $n$ distinct closed points of $X$.

A logarithmic connection on a vector bundle $E$ over $X$ with singularity over $S$ is a first order differential operator 
$$D:E\to K_X\otimes\mathcal{O}_X(S)\otimes E$$ satisfying the Leibnitz identity. The fiber $(K_X\otimes \mathcal{O}_X(S))_P$ 
is identified with $\mathbb{C}$. Given a logarithmic connection $D$, consider the composition
$$E\to K_X\otimes\mathcal{O}_X(S)\otimes E\to (K_X\otimes \mathcal{O}_X(S)\otimes E)_P=E_P.$$ 
This homomorphism of sheaves defines an endomorphism of the fiber $E_P$. This endomorphism is called the 
residue of $D$ at $P$ and is denoted by ${\rm Res}(D,P)$.

A holomorphic connection on $E_*$ is a logarithmic connection $D$ on $E$, singular over $S$, satisfying
\begin{enumerate}
	\item For any $P\in S$, the residue ${\rm Res}(D,P)$ preserves the filtration of $E_P$, and it is semi-simple;
	\item the action of ${\rm Res}(D,P)$ on $F_i(E_P)/F_{i-1}(E_P)$ is multiplication by the corresponding 
	parabolic weight $\lambda_i^P$. ($F_k(E_P)$ denotes the $k$th element in the filtration of $E_P$)
\end{enumerate}

For a representation $\rho:\pi_1(X-S)\to GL(n)$ of the fundamental group of $X-S$ such that the image 
of the generator $\tau_P$ of the isotropy group at $P$ is $\exp(A_P)$  
where $A_P$ is a semisimple element (diagonalizable) for all $P\in S$, one can associate a parabolic bundle 
of degree zero. It is well known that the vector bundle $E_{\rho}$ on $X-S$ admits a 
holomorphic connection since it comes from a representation (see \cite{AMF}). 
This connection on $E_{\rho}$ prolongs to a holomorphic connection on $E_*$ (a logarithmic connection with at worst logarithmic singularities along $S$) whose monodromy representation is $\rho$.
Conversely, a parabolic bundle $E_*$ which admits a holomorphic connection $D$ arises from the monodromy representation $\rho$ of $(E_*,D)$. Thus, one can conclude that a parabolic bundle arises from a representation if it admits a holomorphic connection. Hence, it is enough to prove the existence of a holomorphic connection on a parabolic bundle in order to know if it is associated to a representation.

\begin{remark}
	A parabolic bundle $E'_*$ is said to be a direct summand of $E_*$ if there is another parabolic bundle $E''_*$ such that $E_*$ is isomorphic to $E'_*\oplus E''_*$.	
\end{remark}

\noindent The parabolic analogue of Weil's criterion proved in \cite{IB} and \cite{IBM} precisely answers the condition for the existence of a holomorphic connection as follows. 

\begin{theorem}\label{WeilAnalogue}
	A parabolic vector bundle $E_*$ admits a holomorphic connection if and only if every direct summand of $E_*$ is 
	of parabolic degree zero. Equivalently,
	
A parabolic bundle arises from a representation if and only if every direct summand of $E_*$ is of parabolic degree zero. 
\end{theorem}

\subsection{Generalised Parabolic Sheaves with Parabolic Structures}\label{SecGPB}\hfill 

 In this section, we recall the notion of `` generalised parabolic sheaf'' (GPS) defined in \cite{MSTR} and \cite{XS}.

\noindent Unless otherwise mentioned, $Y$ denotes an irreducible nodal curve over $\mathbb{C}$ with exactly one node $y_0$ and $p:X\rightarrow Y$ denotes its normalisation with $p^{-1}(y_0)=\{x_1,x_2\}$. Let $E$ be a coherent sheaf on $X$, torsion free sheaf outside $\{x_1,x_2\}$ such that $\text{rk }E=r>0$.

\begin{definition}
	A generalised parabolic structure on a sheaf $E$ over the divisor $\{x_1,x_2\}$ is a choice of an $r$-dimensional 
	quotient $Q$ of $E_{x_1}\oplus E_{x_2}$. A sheaf with a generalised parabolic structure is called a GPS.
\end{definition}

\begin{definition}
	A GPS $(E,Q)$ is said to be stable (resp. semistable) if for every proper subsheaf $E'$ such that $E/E'$ is 
	torsionfree outside $\{x_1,x_2\}$ we have 
	\begin{equation*}
		\frac{\deg E'-\dim Q^{E'}}{\text{rk }E'}<(\text{resp. }\leq)\frac{\deg E-\text{rk }E}{\text{rk }E} 
	\end{equation*}
	where $Q^{E'}$ denotes the image of $E'_{x_1}\oplus E'_{x_2}$ in $Q$.
\end{definition}
 
 \begin{definition}
  A generalised parabolic sheaf on $X$ with parabolic structures at finitely many smooth points $p_1,\dots,p_n$ of $X$ is a generalised parabolic sheaf $(E,Q)$ over $X$ where $E$ has parabolic structures at  $p_1,\dots, p_n$. We use $(E_*,Q)$ to denote it.
 \end{definition}

\begin{definition}
	A generalised parabolic sheaf $(E_*,Q)$ with parabolic structures at a finite set of smooth 
  points $p_1,\dots, p_n$ of $X$ is said to be stable (resp. semistable) 
	if for every non-trivial subsheaf $E'$ such that $E/E'$ is torsionfree outside $\{x_1,x_2\}$, we have
	\begin{equation}
		\frac{{\rm par}\deg E'_*-\dim Q^{E'}}{{\rm rk}\ E'}<(\text{resp. }\leq)\frac{{\rm par}\deg E_*-\dim Q}{{\rm rk}\ E}
	\end{equation}
	where $Q^{E'}$ denotes the image of $E'_{x_1}\oplus E'_{x_2}$ in $Q$.
\end{definition}

Note that, $p: X-\{x_1,x_2\}\to Y-{y_0}$ is a bijection. Hence, for any finite subset $I$ of smooth points of $Y$ 
we can identify $I$ with $p^{-1}(I)$. Given a generalized parabolic sheaf $(E_*,Q)$ on $X$ with parabolic 
structures over $I$, one obtains a sheaf $F$ on $Y$ with parabolic structures over $I$ 
which fits into the following exact sequence:
\begin{equation}\label{ExactQPB}
	0\rightarrow F\rightarrow p_*E\rightarrow Q_{y_0}\rightarrow 0
\end{equation}
where $Q_{y_0}$ is the skyscraper sheaf on $Y$ with support at $y_0$ and fibre $Q$.


\begin{remark}\label{BhSun}
 When $E$ is a locally free sheaf, a GPS $(E_*,Q)$ gives rise to an exact sequence
	$$0\to F_1(E)\to E_{x_1}\oplus E_{x_2}\xrightarrow{q} Q\to 0$$
	where $\dim F_1(E)={\rm rk}\ E$. 
	
	\noindent In works of Bhosle, the pair $(E, F_1(E))$ is called
	a Generalised parabolic bundle (GPB). We say a GPB $(E_*,F_1(E))$ is stable (resp. semi stable), if 
	for any subsheaf $E'$ of $E$ such that $E/E'$ is torsion free sheaf, we have
	\begin{equation}
	\frac{{\rm par}\deg E'_*+\dim F_1(E')}{{\rm rk}\ E'}<(\text{resp. }\leq)\frac{{\rm par}\deg E_*+\dim F_1(E)}{{\rm rk}\ E}
	\end{equation}
	where $F_1(E')= E'_{x_1}\oplus E'_{x_2}\cap F_1(E)$. 	
	
	\noindent It it easy to verify that the stability (resp. semi stability) $(E_*,Q)$ is equivalent to the stability (resp. semi stability)
	of $(E,F_1(E))$. 
\end{remark}

\begin{notation}\label{GPBBhosle}
 We also call $(E,Q)$ and  $(E_*, Q)$ a GPB if $E$ is locally free. We denote the composition map $F_1(E)\to E_{x_1}\oplus E_{x_2} \to E_{x_i}$ by $p_i$ and the composition map $E_{x_i} \to E_{x_1}\oplus E_{x_2}\xrightarrow{q_i} Q$ by $q_i$ for $i=1,2$.
\end{notation}

\begin{lemma}\label{GPBcorresp}\hfill
	\begin{enumerate}
		\item Let $(E_*,Q)$ be a GPB over $X$ such that the natural maps $q_i: E_{x_i}\to Q$
		are isomorphisms and let $F$ be the associated sheaf on $Y$. Then $F$ is a locally free sheaf.
		\item If $F_*$ is a parabolic vector bundle on $Y$, there is a unique GPB $(E_*,Q)$ 
		which gives $F_*$. In fact, $E=p^*F$.
	\end{enumerate}
\end{lemma}

\begin{proof}
	See \cite[Lemma 4.6]{MSTR} or \cite[Lemma 2.1]{XS}.
\end{proof}
\begin{proposition}\label{GPB}	Let $(E_*,Q)$ be a GPS over $X$ and $F_*$ denotes its associated parabolic sheaf on $Y$. Then 
 $F_*$ is semistable if and only if $(E_*,Q)$ is semistable. Moreover, one has
\begin{enumerate}
		\item If $F_*$ is a parabolic stable bundle, then $(E_*,Q)$ is stable GPB;
		\item If $(E_*,Q)$ is stable, then $F_*$ is  stable.
		\end{enumerate}
\end{proposition}
 
 \begin{proof}
 		See  \cite[Proposition 4.7]{MSTR} or \cite[Lemma 2.2]{XS}.
  \end{proof}

\begin{proposition}\label{bijection}
	Let $R$ denote the isomorphism classes of parabolic vector bundles $F_*$ on $Y$ (for simplicity, assume parabolic structure at a single point $P$) with ${\rm rk}\ F=n$ and ${\rm par}\ \deg F_*=0$. Now, let $T$ denote the set of isomorphism classes of GPBs $(E_*,Q)$ on $X$ with the same parabolic structure at $P$ such that ${\rm rk }\ E=n$, ${\rm par}\deg E_*=0$, and $\dim Q=n$ with the maps $E_{x_j}\rightarrow Q$ isomorphisms. Then, there is a bijective correspondence $f:T\rightarrow R$. If $f(E_*,Q)=F_*$, then $E=p^*F$.
	
	Under the correspondence $f$, semistable (resp. stable) GPBs on $X$ correspond to parabolic semistable (resp. parabolic stable) bundles on $Y$.  
\end{proposition}

\begin{proof}
	Follows from Lemmas \ref{GPBcorresp} and \ref{GPB}.
\end{proof}

\section{GPBs with Parabolic Structures and Representations}

As in Proposition \ref{bijection}, let $T$ denote the set of isomorphism classes of GPBs $(E_*,Q)$ on $X$ with a 
fixed parabolic structure at $P$ such that ${\rm rk }\ E=n$, ${\rm par}\deg E_*=0$, and $\dim Q=n$ with the
maps $E_{x_j}\rightarrow Q$ isomorphisms. Let $T_0$ denote the subset of $T$ consisting of GPBs for which every 
direct summand of $E_*$ is of parabolic degree zero. Our aim is to give a correspondence between 
GPBs $(E_*,Q)$ in $T_0$ and a subset of representations $\zeta:\pi_1(X-P)*\mathbb{Z} \to GL(n, \mathbb{C})$.

Given a representation $\zeta:\pi_1(X-P)*\mathbb{Z}\rightarrow GL(n,\mathbb{C})$ such that its 
restriction $\zeta|_{\pi_1(X-P)}=(\zeta_{X},A_P)$ (i.e., it satisfies $\zeta(\tau_P)=\exp(A_P)$ for some 
diagonalizable matrix $A_P$), we get a parabolic bundle $(E_{\zeta})_*$ on $X$ of zero parabolic degree. 
From Theorem \ref{WeilAnalogue}, we get that every direct summand of $(E_{\zeta})_*$ has parabolic degree zero.

Again, we consider $X$ as a compact Riemann surface and follow the same notation as in \S \ref{rep}.
Let $1\in\mathbb{Z}$ be the generator and let $g$ denote the image of 1 under $\zeta$ in $GL(n,\mathbb{C})$, 
i.e., $g=\zeta(1)$. Fix points $z_1,z_2$ in $H$ lying over $x_1,x_2$, respectively, and 
identify $(E_{\zeta})_{x_1}$ and $(E_{\zeta})_{x_2}$ with $\mathbb{C}^n$ to realise $g$ as an 
isomorphism, say $\sigma$, of $(E_{\zeta})_{x_1}$ onto $(E_{\zeta})_{x_2}$. Let $F_1(E_{\zeta})$ 
denote the graph of $\sigma$ in $(E_{\zeta})_{x_1}\oplus (E_{\zeta})_{x_2}$ and define $Q_{\zeta}$ 
to be the quotient $Q_{\zeta}:=(E_{\zeta})_{x_1}\oplus (E_{\zeta})_{x_2}/F_1(E_{\zeta})$. 
Then $Q_{\zeta}$ is an $n$-dimensional quotient, and  we obtain a GPB $\big((E_{\zeta})_*,Q_{\zeta}\big)$ 
with parabolic structures at $P$ and parabolic degree zero, associated to $\zeta$. 
In fact, $\big((E_{\zeta})_*,Q_{\zeta}\big)\in T_0$. 

\begin{theorem}\label{GPBrep}
	A GPB $(E_*,Q)$ in $T$ is associated to a representation $\zeta$ of $\pi_1(X-P)*\mathbb{Z}$ if and only if 
	every direct summand of $E_*$ is of parabolic degree zero.
\end{theorem}

\begin{proof}
	Given a GPB $(E_*,Q)$ on $X$ with parabolic structures at $P$ where $E_*$ is such that every direct summand is of 
	parabolic degree 0, using Theorem \ref{WeilAnalogue} we get a representation $\zeta_X$ of $\pi_1(X-P)$ 
	such that $E_*=(E_{\zeta_{X}})_*$. Since the natural maps $q_1: E_{x_1} \rightarrow Q$ and $q_2: E_{x_2} \rightarrow Q$ 
	are assumed to be isomorphisms, we have an isomorphism $q_2^{-1}\circ q_1:(E_{\zeta})_{x_1}\rightarrow(E_{\zeta})_{x_2}$ 
	and hence an element $g$ of $GL(n,\mathbb{C})$. Define a representation $\zeta:\pi_1(X-P)*\mathbb{Z}\rightarrow GL(n,\mathbb{C})$ 
	by $\zeta|_X=\zeta_{X}$ on $\pi_1(X-P)$ and $\zeta(1)=g$ where $1$ denotes the generator of $\mathbb{Z}$. It is easy to see 
	that $(E_*,Q)\approx \big((E_{\zeta})_*,Q_{\zeta}\big)$. 
\end{proof}

Using Corollary \ref{fungrpcor}, any representation $\zeta$ of 
$\pi_1(X-P)*\mathbb{Z}$ can also be regarded as a representation of $\pi_1(Y-P)$ in $GL(n, \mathbb{C})$.
Therefore, we can associate to it a parabolic bundle $(F_\zeta)_*=f\left(\big((E_{\zeta})_*,Q_{\zeta}\big)\right)$.
Henceforth, a parabolic bundle $(F_{\zeta})_*$ on $Y$ arising from a representation $\zeta$ of $\pi_1(Y-P)$ means that $\zeta|_{\pi_1(X-P)}=(\zeta_{X},A_P)$ (i.e., it satisfies $\zeta(\tau_P)=\exp(A_P)$ for some 
diagonalizable matrix $A_P$ as in \S \ref{rep}) and $(F_\zeta)_*$ is as constructed above.

\begin{lemma}\label{inequalities}
	Let $(E_*,Q)\in T$ be a GPB on $X$. Then for any non-trivial proper subbundle $E'$ of $E$. We have the following inequalities:
	\begin{enumerate}
		\item $\dim Q^{E'}\geq \text{rk }E'$
		\item $\dim Q^{E'}\leq \text{rk }E$		
	\end{enumerate}
\end{lemma}

\begin{proof}
	(1) A GPB $(E_*,Q)\in T$ gives rise to an exact sequence
	$$0\to F_1(E)\to E_{x_1}\oplus E_{x_2}\xrightarrow{q} Q\to 0$$
	where $\dim F_1(E)={\rm rk}\ E$ and the projections $p_i: F_1(E)\to E_{x_i}$ are isomorphisms. For a subbundle $E'\subset E$ if we consider $F_1(E'):=F_1(E)\cap (E'_{x_1}\oplus E'_{x_2})$ and $Q^{E'}:=q(E'_{x_1}\oplus E'_{x_2})$, we get another exact sequence
	$$0\to F_1(E')\to E'_{x_1}\oplus E'_{x_2}\xrightarrow{q} Q^{E'}\to 0.$$
	Since $p_i: F_1(E)\to E_{x_i}$ are isomorphisms, $p_i$ maps $F_1(E')$ injectively into $E'_{x_i}$. It follows that $\dim F_1(E')\leq{\rm rk}\ E'$. Substituting $\dim F_1(E')=2{\rm rk}\ E'-\dim Q^{E'}$, we get
	$$\dim Q^{E'}\geq \text{rk }E'.$$ 
	
	(2) Since $q_i:E_{x_i}\to Q$ is an isomorphism and $Q^{E'}\subset Q$, $\dim Q^{E'}\leq\dim Q={\rm rk}\ E$. 	
\end{proof}


\begin{proposition}\label{stabilitycorresp}
	\hfill
	\begin{enumerate}
		\item Let $(E_*,Q)\in T$. If $E_*$ is a parabolic semistable (resp. parabolic stable) bundle, then $(E_*,Q)$ is a semistable (resp. stable) GPB.
		\item Let $F_*$ be a parabolic bundle on $Y$ of parabolic degree zero. If $p^*(F)_*$ is parabolic 
		semistable (resp. parabolic stable), then $F_*$ is parabolic semistable (resp. parabolic stable).
		\item If $\zeta_{X}=\zeta|_{\pi_1(X-P)}$ is a unitary (resp. irreducible unitary) representation 
		of $\pi_1(X-P)$, then the parabolic bundle $(F_{\zeta})_*$ on $Y$ associated to $\zeta$ is parabolic
		semistable (resp. parabolic stable).
	\end{enumerate}
\end{proposition}

\begin{proof}
	(1) $E_*$ is parabolic semistable (resp. parabolic stable) implies for all subbundles $E'\subset E$, ${\rm par} \deg E'_*\leq 0$ (resp. $<0$) and from Lemma \ref{inequalities} we have $\dim Q^{E'}\geq {\rm rk}\ E'$. Thus we get
	\begin{align*}
		{\rm par}\deg E'_*-\dim Q^{E'} & \leq -{\rm rk}\ E'\\
		{\rm i.e.,}\ \frac{{\rm par}\deg E'-\dim Q^{E'}}{{\rm rk}\ E'} &\leq -1\\
		& =  	\frac{{\rm par}\deg E_*-\dim Q}{{\rm rk}\ E}
	\end{align*}
	implying $(E_*,Q)$ is semistable. The proof similarly holds when $\leq$ is replaced by $<$.
	
	(2) Given $F_*$, we have a unique GPB $(E_*,Q)$ which yields $F$ in \eqref{ExactQPB} where $E=p^*F$. Since $p^*F_*$ is parabolic semistable from the first part of the proposition we have $(E_*,Q)$ is semistable and we get $F_*$ is parabolic semistable.
	
	(3) If $\zeta_{X}=\zeta|_{\pi_1(X-P)}$ is unitary (resp. irreducible unitary), using Theorem \ref{MS}  
	we get $(E_{\zeta})_*$ is parabolic semistable (resp. parabolic stable). It follows from (2) that $(F_\zeta)_*$ is parabolic semistable (resp. stable) since $(E_{\zeta})_*=p^*(F_{\zeta})_*$.
\end{proof}

 We summarize our observations in the following theorems.

 \begin{theorem}
 	Let $g(X)\geq 2$. A parabolic bundle $F_*$ with parabolic structures at $P$ of parabolic degree zero on an irreducible nodal curve $Y$ is associated to a representation $\zeta$ of $\pi_1(Y-P)$ if and only if every direct summand of its pullback to the normalisation $X$ has parabolic degree zero.
 	
 	If $\zeta_{X}$, the restriction of $\zeta$ to $\pi_1(X-P)$, is unitary (resp. irreducible unitary) then $F_*$ is parabolic semistable (resp. parabolic stable).
 \end{theorem}

\begin{proposition}\label{anyrank1node}
	Let $g(X)\geq 2$. For every integer $n\geq 4$,
	there exists a stable parabolic bundle  of rank $n$ and 
	parabolic degree zero on the nodal curve which is not associated to a representation of $\pi_1(Y-P)$.
\end{proposition}
\begin{proof}
         If a parabolic vector bundle $F_*$ on $Y$ is associated to a representation of $\pi_1(Y-P)$ then 
          $(p^*F)_* = E_*$ is a direct sum of indecomposable parabolic bundles of degree zero. Therefore, in view of 
          Proposition $2.23$ and Remark $2.21$, it suffices to show that there exists a stable GPB $(E_*, F_1(E))$ on $X$ such that 
          $E_*\in T-T_0$ with $r(E) \geq 4$.
          
	Given $n\geq 4$, we can write $n=2m$ or $n=2m+1$ for some $m>1$ depending on if $n$ is even or odd. We discuss the case of $n=2m$ in detail, and for odd ranks, the necessary modifications are described. Assume 
	that $n=2m$.
	
	We can always find  parabolic bundles $E_*$ and $F_*$  of rank $m$ such that $E_*$ is stable, 
	$F_*$ is $(1,0)$-stable, $0< {\rm par}\deg E_*\leq 1$, and $-1\leq {\rm par}\deg F_*< 0$ so that
	$G_*=E_*\oplus F_*$ is parabolic bundle with the given parabolic structure at $P\in X$.
	Choose $F_1(G)$ so that $\dim F_1(E)=0$. For example, if $( e_1 ,\dots, e_m)$ is a basis of $E_{x_1}$,
	$(e_{m+1},\dots , e_{2m})$ a basis of $E_{x_2}$, $( g_1 ,\dots, g_m)$ is a basis of $F_{x_1}$ 
and $(g_{m+1},\dots , g_{2m})$ a basis of $F_{x_2}$ then one can take $F_1(F)$ spanned by 
$(e_1+g_{m+1},e_2+g_{m+2},\dots e_m+g_{2m}, e_{m+1}+g_1,e_{m+2}+g_2 \dots e_{2m}+g_m)$.

{\bf Claim:} $(E_*,F_1(E))$ is a stable GPB. 

We show that for any subbundle $N\subset G$ such that $G/N$ is a torsion free sheaf,
	$$\frac{{\rm par}\deg N_*+\dim F_1(N)}{{\rm rk}\ N}<\frac{{\rm par}\deg G_*+\dim F_1(G)}{{\rm rk}\ G},$$
	i.e., $ {\rm par}\deg N_*<{\rm rk}\ N - \dim F_1(N)= \dim Q^N- {\rm rk}\ N.$
	Let $\phi: N\to F$ be the composition $N\hookrightarrow G\to \frac{G}{E}\approx F$. 
	Let  $K\subseteq N$ be the kernel and $I$ the image of the morphism $\phi$. 
		It is clear that  ${\rm rk}(N)={\rm rk}(K)+{\rm rk}(I)$ and $K \subseteq E$. 
	Also, the stability of  $E_*$ implies that ${\rm par}\deg K_*< 1$ 
	whenever $0<{\rm rk}\ K<{\rm rk}\ E$, and ${\rm par}\deg K_*\leq  1$ if ${\rm rk}\ K={\rm rk}E$. 
	Now, we consider all possibilities for $I$ and prove the claim in each case. 
	
	Case (i): ${\rm rk}\ I=m$. This implies $I=F,\ {\rm par}\deg I_*={\rm par}\deg F_*<0,$ and ${\rm rk}\ K<m={\rm rk}\ E$ as $N$ is a 
	proper subbundle of $G$. Since ${\rm par}\deg N_*\leq{\rm par}\deg K_*+{\rm par}\deg I_*<0$, applying Lemma \ref{inequalities}, we have ${\rm par}\deg N_*<\dim Q^N-{\rm rk}(N)$.
	
	Case (ii): $0<{\rm rk}\ I<m$. Since $F_*$ is $(1,0)$-stable, we get 
	$$\frac{{\rm par}\deg I_*+1}{{\rm rk}\ I}<\frac{{\rm par}\deg F_*+(1-0)}{{\rm rk}\ F},$$
	which implies that ${\rm par}\deg I_*<-1$. Thus, ${\rm par}\deg K_*\leq  1$ and Lemma \ref{inequalities} imply that
	$${\rm par}\deg N_*<0\leq\dim Q^N-{\rm rk}(N).$$
	
	Case (iii): ${\rm rk}\ I=0$. This implies $K=N\subseteq E$ and $\dim F_1(N)=0$ ($\dim Q^N=2{\rm rk}\ N$)
	
	(a) If ${\rm rk}\ N={\rm rk}\ E$, then ${\rm par}\deg N_*\leq 1<m={\rm rk}\ N=\dim Q^N-{\rm rk}(N)$.
	
	(b) If $0<{\rm rk}\ N<{\rm rk}\ E$, since $E_*$ is stable and ${\rm rk}(N)\geq 1$  we get 
	$${\rm par}\deg N_*<1 \leq {\rm rk}(N)=\dim Q^N-{\rm rk}(N).$$ 
	
	For $n=2m+1$, we take $E_*$ to be $(0,1)$ stable, ${\rm rk}\ E=m+1$ and $F_1(G)$ is chosen so that $\dim F_1(E)=1$ (this is the
best possible in this case) in the proof above. 
	
	Note that the $(0,1)$ stability of $E_*$ implies that $E_*$ is stable and cases (i), (ii) and (iii)(a)
	follow similarly as above. We discuss the case (iii)(b). 
	
	Case (iii): ${\rm rk}\ I=0$. This implies $K= N\subseteq E$ and 
	$\dim F_1(N)\leq1 $ ($2{\rm rk}\ N-1\leq\dim Q^N\leq 2{\rm rk}\ N$).
	
         (b) In this case, there arises a possibility where ${\rm rk}\ N=1$ and $\dim Q^N=1$ which implies
	 $\dim Q^N-{\rm rk}(N)=0$; the above proof of Case (iii)(b) fails.
		 The $(0,1)$ stability of $E_*$ resolves this issue. Note that,
	 $E_*$ is $(0,1)$-stable implies that ${\rm par}\deg N_*<0$. Hence we get
	 $${\rm par}\deg N_*<0  \leq {\rm rk}(N)-1\leq\dim Q^N-{\rm rk}(N).$$ 
\end{proof}

\begin{remark}
	Using our method in the odd rank case, one can give an easy  proof of \cite[Proposition 3.3]{UB3}. 
\end{remark}

\begin{remark}
 Proposition \ref{stabilitycorresp}$(3)$ is also true if the curve has finitely many nodes. Also, one considers 
 the representations of $\pi_1\left(Y-\{p_1,\dots, p_n\}\right)$ to construct parabolic bundles on 
 $Y$ with parabolic structures at $p_1,\dots ,p_n$.
In short, all results in this section are also true in the case of finitely many nodes and parabolic points.
\end{remark} 

\begin{remark}
  Proposition \ref{anyrank1node} also holds for ranks two and three, if the curve has more than one node.
In this case, the proof follows similarly as above except for case (iii)(a) and is resolved for 
curves with at least two nodes.

\end{remark}

\section{Codimension Computations}

Assume that $X$ and $Y$ are as defined in Section \ref{SecGPB}. Let $\mathcal{U}_Y(n,d,I, \alpha,k)$ denote the $S$-equivalence classes of semistable parabolic sheaves of rank $n$ and degree $d$  with parabolic structures over $I$ of 
weights $\alpha(y)=(\alpha_1(y),\dots, \alpha_{r_y}(y))$ 
and multiplicities $k(y)=( k_1(y),\dots, k_{r_y}(y)$ for each $y\in I$. Let $\mathcal{U}^s_Y(n,d,I, \alpha,k)$
be the open dense subset of $\mathcal{U}_Y(n,d,I, \alpha,k)$ consisting of stable parabolic sheaves.
We write $\mathcal{U}_Y:=\mathcal{U}_Y(n,d,I, \alpha,k)$ and  $\mathcal{U}_Y^s:=\mathcal{U}_Y^s(n,d,I, \alpha,k)$

Similarly, we denote by $\mathcal{U}_X$,  the $S$-equivalence classes of 
semistable parabolic bundles over $X$ and by $\mathcal{U}^s_X$ the open subset
of $\mathcal{U}_X$ consisting of stable parabolic bundles.
We recall the construction  of $\mathcal{U}_X$ from \cite{XS}. 

Let $\widetilde{\mathcal{Q}}$ be the Quot scheme of coherent quotients of $\mathcal{O}_X^N$ where $N=d+n(1-g(X))$. 
There exists a universal quotient sheaf $\mathcal{F}$ on $X\times \widetilde{\mathcal{Q}}$, flat 
over $\widetilde{\mathcal{Q}}$ such that $\mathcal{O}_{X\times\widetilde{\mathcal{Q}}}\to \mathcal{F}\to 0$. 
Let $\mathcal{F}_x$ be the sheaf obtained by restricting $\mathcal{F}$ to $\{x\}\times \widetilde{\mathcal{Q}}$ 
and let $Flag_{\vec{k}(x)}\mathcal{F}_x$ denote the relative flag scheme of type $\vec{k}(x)$. Let $\widetilde{\mathcal{R}}$
be the fibre product over $\widetilde{\mathcal{Q}}$ defined by
$$\widetilde{\mathcal{R}}=\times_{\underset{x\in I}{\widetilde{\mathcal{Q}}}}Flag_{\vec{k}(x)}(\mathcal{F}_x)$$
so that we get a flag bundle $\widetilde{\mathcal{R}}\to\widetilde{\mathcal{Q}}$.

Let $\widetilde{\mathcal{R}}^s$ (resp. $\widetilde{\mathcal{R}}^{ss}$) be the open subscheme corresponding to 
stable (resp. semistable) parabolic bundles, which is generated by global sections and whose first cohomology 
vanishes when $d$ is large enough. The variety $\mathcal{U}_X$ is the good quotient of $\widetilde{\mathcal{R}}^{ss}$ 
by $SL(N)$ acting through $PGL(N)$.

\noindent We have the following estimates on codimension.

\begin{proposition}\label{codimparabolic}
	With notations as above, 
	$${\rm codim}(\widetilde{\mathcal{R}}-\widetilde{\mathcal{R}}^{s},\widetilde{\mathcal{R}})\geq (n-1)(g(X)-1)+1\,.$$
\end{proposition}
\begin{proof}
 A proof of this follows from the fact that
 $${\rm codim}(\widetilde{\mathcal{R}}-\widetilde{\mathcal{R}}^{s},\widetilde{\mathcal{R}})= {\rm codim}(\widetilde{\mathcal{R}^{ss}}-\widetilde{\mathcal{R}}^{s},\widetilde{\mathcal{R}^{ss}}),$$
and  \cite[Proposition 5.1(1)]{XS}.
 \end{proof}

Next, we consider generalised parabolic sheaves $(E_*,Q)$ on $X$ with parabolic structures over $I$. 
One defines an $S$-equivalence of GPS on $X$. By \cite[Theorem 1.3]{XS}, there exists a (coarse) moduli space 
$\mathcal{P}^s$ of stable GPS on $X$, which is a smooth variety. We have an open immersion 
$\mathcal{P}^s\hookrightarrow  \mathcal{P}$, where $\mathcal{P}$ is the space of $S$-equivalence classes of 
semistable GPS on $X$, which is a reduced, normal, projective variety 
with rational singularities. Note that, every semistable GPS $(E',Q')$ with ${\rm rk}(E')>0$ is $S$-equivalent to a semistable GPB $(E,Q)$ (\cite[Lemma 2.5]{XS}). 
There exists a finite surjective morphism $\phi : \mathcal{P}\to \mathcal{U}_Y$ and is the normalisation of $\mathcal{U}_Y$. For more details, see \cite[Proposition 2.1]{XS}.

Define $\widetilde{\mathcal{R}}'$ by
$$\widetilde{\mathcal{R}}':=Gr(n,\mathcal{F}_{x_1}\oplus\mathcal{F}_{x_2})\times_{\widetilde{\mathcal{Q}}}\widetilde{\mathcal{R}}$$
where $\mathcal{F}$ is the universal quotient 
$\mathcal{O}_{X\times\widetilde{\mathcal{Q}}}\to \mathcal{F}\to 0$ on $X\times \widetilde{\mathcal{Q}}$.

Let $\widetilde{\mathcal{R}}'^{ss}$ (resp. $\widetilde{\mathcal{R}}'^{s}$) denote the open subset of 
$\widetilde{\mathcal{R}'}$ consisting of semistable GPS (resp. stable GPS). The space $\mathcal{P}$ 
(resp. $\mathcal{P}^s$) is the  GIT quotient of 
$\widetilde{\mathcal{R}}'^{ss}$ (resp. $\widetilde{\mathcal{R}}'^{s}$) by $SL(N)$ acting through $PGL(N)$.

Let $\rho$ denote the natural projection $\rho:\widetilde{\mathcal{R}}'\to\widetilde{\mathcal{R}}$. 
Then $\widetilde{\mathcal{R}}'\to\widetilde{\mathcal{R}}$ is a grassmannian bundle over $\widetilde{\mathcal{R}}$.
Let  $\widetilde{\mathcal{H}}'$ be the open subscheme of $\widetilde{\mathcal{R}}'$ corresponding to
GPSs $(E_*,Q)$ such that the maps $p_1\colon E_{x_1}\to Q$ and $p_2\colon E_{x_2}\to Q$ are isomorphisms. 
Let  $\widetilde{\mathcal{H}}'^{ss}\subset  \widetilde{\mathcal{H}}'$ 
(resp. $\widetilde{\mathcal{H}}'^{s}\subset  \widetilde{\mathcal{H}}'$) denote the subset consisting of semistable (resp. stable) points.
Let $\overline{\mathcal{H}}'^s\subset \widetilde{\mathcal{H}}'$ 
be the subset corresponding to 
GPSs $(E_*,Q)$ where the underlying parabolic bundle $E_*$ is parabolic stable. 
Let  $\mathcal{P}'$ and $\overline{\mathcal{P}}^s$ denote the open subsets of $\mathcal{P}$ which 
are quotients of $\widetilde{\mathcal{H}}'^{ss}$ and $\overline{\mathcal{H}}'^s$, respectively.

\begin{proposition}\label{G}
	With notations as above, we have 
	$${\rm codim}(\widetilde{\mathcal{H}}'-\overline{\mathcal{H}}'^{s},\widetilde{\mathcal{H}}')\geq(n-1)(g(X)-1)+1.$$

\end{proposition}

\begin{proof}
	 Since $\rho:\widetilde{\mathcal{H}}'\to \widetilde{\mathcal{R}}$ is a fibre bundle with fibres isomorphic 
	 to $GL(n)$, we get
	 $${\rm codim}(\rho^{-1}(Z),\widetilde{\mathcal{H}}')={\rm codim}(Z,\widetilde{\mathcal{R}})$$ for 
	any closed subset $Z\subset \widetilde{\mathcal{R}}$.
	Now, from Proposition \ref{codimparabolic} and Proposition \ref{stabilitycorresp} it follows that
	${\rm codim}(\widetilde{\mathcal{H}}'-\overline{\mathcal{H}}'^{s},\widetilde{\mathcal{H}}')=
	{\rm codim}(\rho^{-1}(\widetilde{\mathcal{R}}-\widetilde{\mathcal{R}}^{s}),\widetilde{\mathcal{H}}')=
	{\rm codim}(\widetilde{\mathcal{R}}-\widetilde{\mathcal{R}}^{s},\widetilde{\mathcal{R}})\geq(n-1)(g(X)-1)+1.$
	
\end{proof}

Taking GIT quotients, we get the following corollary. 

\begin{corollary}\label{GG}
	With notations as above, we have
	$${\rm codim}(\mathcal{P}'-\overline{\mathcal{P}}^{s},\mathcal{P}')\geq(n-1)(g(X)-1)+1.$$
\end{corollary}

\begin{proof}
  A proof follows from the fact 
  $${\rm codim}(\widetilde{\mathcal{H}}'-\overline{\mathcal{H}}'^{s},\widetilde{\mathcal{H}}')=
{\rm codim}(\widetilde{\mathcal{H}}'^{ss}-\overline{\mathcal{H}}'^s,\widetilde{\mathcal{H}}'),$$
                           and from Proposition \ref{G}.         
 \end{proof}

Let  $\mathcal{U}'_Y$  be the open dense subvariety of $\mathcal{U}_Y$ corresponding to parabolic vector bundles and let  $\overline{\mathcal{U}}_Y'^s$ be the  subset of $\mathcal{U}'_Y$ corresponding to parabolic bundles $F_*$ such that $p^*(F_*)$ is parabolic stable.

\begin{theorem}\label{codimpullback}
	With notations as above, we have 
	$${\rm codim}(\mathcal{U}'_Y-\overline{\mathcal{U}}_Y^{s},\mathcal{U}'_Y)\geq(n-1)(g(X)-1)+1.$$
\end{theorem}

\begin{proof}
	Since $\mathcal{P}'$ and $\mathcal{U}_Y'$ are isomorphic, the theorem follows from Corollary \ref{GG}.
\end{proof}

We denote by  $U_Y$  the moduli space of $S$-equivalence 
classes of semistable parabolic sheaves of parabolic degree zero. Similarly, denote by  $U'_Y$, the open dense subvariety of $U_Y$ corresponding to parabolic vector bundles and by  $\overline{U}_Y'^s$, the  subset of $U'_Y$ 
corresponding to parabolic bundles $F_*$ such that $p^*(F_*)$ is parabolic stable. 

 \begin{theorem}\label{Repbijection}
	Define the set ${\rm Rep}$ by 
	\begin{equation*}
		{\rm Rep}:= 	\left\{
		\begin{aligned}
			& {\rm Equivalence \ classes \ of \ representations} \ (\rho,A)\colon \pi_1(Y-P)\to {\rm GL(n,\C)} \\
			& {\rm such \ that} \ \rho_X=\rho|_{\pi_1(X-P)} \ {\rm is \ irreducible \ unitary}
		\end{aligned}
		\right\}\,.	
	\end{equation*}
	Then there is a bijective correspondence between ${\rm Rep}$ and $\overline{U}_Y'^s$.	
\end{theorem}

\begin{proof}
	If $\rho_X$ is irreducible unitary, then $(E_{\rho})_*$ on $X$ and hence $(F_{\rho})_*$ on $Y$ are parabolic 
	stable (Proposition \ref{stabilitycorresp}). Also, $(E_{\rho})_*$ is the pullback of $(F_{\rho})_*$.
	
	If $F_*$ is a parabolic vector bundle such that $p^*(F_*)$ is parabolic stable, by \cite{VMCS}, there 
	is an irreducible unitary representation $\rho_X$ of $\pi_1(X-P)$ such that $p^*(F_*)=(E_{\rho_X})_*$. 
	Now $F_*$ can be obtained from $(E_{\rho_X})_*$ by identifying the fibres of $E_{\rho_X}$ at $x_1,x_2$ by 
	an isomorphism $\sigma$. Choosing $\sigma$ is equivalent to choosing an element $g\in {\rm GL}(n,\C)$. 
	Now one can define a representation $\rho:\pi_1(X-P)*X\to {\rm GL}(n,\C)$ by $\rho|_{\pi_1(X-P)}=\rho_X$ 
	and $\rho(1)=g$ as described in Theorem \ref{GPBrep}. Then $F_*=(F_{\rho})_*$
\end{proof}

\begin{theorem} 
	Let $Y$ be a complex nodal curve with $g(X) \ge 2$ and let $P\in Y$ be a smooth point. 
	The subset of $U'_Y$ consisting of vector bundles which arise from irreducible unitary 
	representations of the  fundamental group of $\pi_1(Y-P)$ has complement of codimension at least $2$.
\end{theorem}

\begin{proof}
	The bijection in Theorem \ref{Repbijection} is, in fact, a homeomorphism and can be seen using the fact that 
	Mehta-Seshadri correspondence is a homeomorphism. Hence the theorem follows from Theorem \ref{codimpullback}(2).
\end{proof}

\section*{Acknowledgments}
The authors are extremely grateful to Prof. Usha N. Bhosle for a careful reading of the 
paper and for providing insightful comments and detailed suggestions  which helped in improving the manuscript considerably. The second named author is supported by the SERB Early Career Research Award (ECR/2016/000649) by the Department of Science \& Technology (DST), Government of India. 

\bibliographystyle{plain}

\end{document}